\documentclass{amsart}

\newcommand{\Q}{\mathbb{Q}}
\newcommand{\C}{\mathbb{C}}

\newcommand{\R}{\mathbb{R}}
\newcommand{\N}{\mathbb{N}}

\newcommand{\dom}{\operatorname{dom}}
\newcommand{\ran}{\operatorname{ran}}

\newcommand{\norm}[1]{\left\| #1 \right\|}

\newcommand{\unit}{\mathbf{1}}
\newcommand{\ratpoly}{\mathfrak{p}}
\newcommand{\Ind}{\operatorname{Ind}}

\theoremstyle{theorem}
\newtheorem{theorem}{Theorem}[section]
\newtheorem{lemma}[theorem]{Lemma}

\newtheorem{corollary}[theorem]{Corollary}

\newtheorem{proposition}[theorem]{Proposition}

\theoremstyle{definition}
\newtheorem{definition}[theorem]{Definition}

\numberwithin{equation}{section}

\begin{document}
\title{Evaluative presentations}
\author{Timothy H. McNicholl}
\address{Department of Mathematics\\
Iowa State University\\
Ames, Iowa 50011}

\email{mcnichol@iastate.edu}

\begin{abstract}
We study presentations of $C^*(X)$ that are evaluative over a presentation of $X$ in that 
$(f,p) \mapsto f(p)$ is computable.  We prove existence-uniqueness theorems for such presentations.
We use our methods to prove an effective Banach-Stone Theorem for unital commutative $C^*$ algebras.
We also apply our results to the computable categoricity of $C^*$ algebras and compact Polish spaces.
\end{abstract}
\maketitle

\section{Introduction}\label{sec:intro}

The study of computable presentations of 
metric structures was initiated by Melnikov in \cite{Melnikov.2013}.  
The general goal of this line of inquiry is to extend the framework of computable structure
theory from countable algebraic structures to uncountable metric structures.  
A number of results have been obtained for computable presentations of metric spaces
(see e.g. \cite{Melnikov.Nies.2013}, \cite{Melnikov.2013}) and Banach spaces (see 
e.g. \cite{Melnikov.2013}, \cite{Melnikov.Ng.2014}, \cite{McNicholl.2017}, 
\cite{Clanin.McNicholl.Stull.2019}, \cite{Brown.McNicholl.2020}, 
\cite{Brown.McNicholl.Melnikov.2020}, \cite{Baz.H-Trainor.Melnikov.2021+}).  Recently, A. Fox initiated the study of 
computable presentations of $C^*$ algebras \cite{Fox.2022+}.   

Here, we focus on computable presentations of spaces of the form $C^*(X)$ (with $\C$ as the 
field of scalars).  
Informally speaking, a computable presentation of such a space defines a 
dense subset of the space with respect to which one can compute the 
norm, addition, multiplication, scalar multiplication, and involution.  
However, a computable presentation of $C^*(X)$ may not provide for the computation of 
the evaluation map $(f,p) \mapsto f(p)$.  
We are thus led to consider comnputable presentations of $C^*(X)$ that are \emph{evaluative} 
over a computable presentation $X^\#$ in the sense that 
the evaluation map can be computed with respect to these presentation.  
A formal definition of this concept will be given in Section \ref{sec:prelim}.  

Building on the recent work of Burton, Eagle, Fox et. al., we prove the following 
theorems about evaluative presentations \cite{BEFGHMMT.2024+}.

\begin{theorem}\label{thm:main.1}
Suppose $X$ is a compact Polish space. 
If $C^*(X)^\#$ is a computable presentation, then, up to computable homeomorphism,
there is a unique presentation of $X$ over which $C^*(X)^\#$ is evaluative.
\end{theorem}

\begin{theorem}\label{thm:main.2}
 If $X^\#$ is a computable presentation of a compact Polish space $X$, and if there is a computable
 presentation of $C^*(X)$ that is evaluative over $X^\#$, then $X^\#$ is computably compact.
\end{theorem}

\begin{theorem}\label{thm:main.3}
If $X^\#$ is a computably compact presentation of a Polish space, then, up to computable isometric 
isomorphism, there is a unique computable presentation of $C^*(X)$ that is evaluative over $X^\#$. 
\end{theorem}

From these results, and the methods used to prove them, we obtain a number of consequences.  
In particular, we show that $C^*(X)$ is computably categorical if and only if 
any two computably compact presentations of $X$ are computably homeomorphic.  
We also prove an effective version of the Banach-Stone Theorem for commutative 
unital $C^*$ algebras.  

The paper is organized as follows.  We summarize required background from analysis and computability in 
Section \ref{sec:back}.  Section \ref{sec:prelim} presents preliminary results about evaluative presentations.  We then prove the main theorems in Section \ref{sec:proofs.main}. 
We consider the uniformity of our results in Section \ref{sec:unif}.  
Section \ref{sec:appl} presents applications of our results.  

\section{Background}\label{sec:back}

We begin with some basic notation and terminology from functional analysis.  

When $A$ is a unital $C^*$ algebra,   
let $\unit_A$ denote the unit of $A$.

Suppose $T : C^*(X) \rightarrow C^*(Y)$ is an isometric isomorphism. 
It follows from the Banach-Stone Theorem that there is a unique homeomorphism 
$\psi : Y \rightarrow X$ so that $T(f) = f \circ \psi$ for all $f \in C^*(X)$.
We call $\psi$ the \emph{spatial realization} of $T$.

We assume the reader is familiar with the concepts, methods, and terminology of 
computable analysis as expounded in \cite{Brattka.Hertling.2021}.  We also assume
knowledge of the rudiments of computability theory as presented in \cite{Cooper.2004}.
For the basic computability 
theory of operator algebras, we refer the reader to \cite{Fox.2022+} and \cite{BEFGHMMT.2024+}.  
For a treatment of computably compact Polish spaces, we refer the reader to 
\cite{Downey.Melnikov.2023}; see also \cite{Brattka.Hertling.2021}.

Suppose $B$ and $B'$ are rational open balls of a presented metric space $X^\#$.  
Let $r,r'$ be rational numbers and let $p.p'$ be distinguished points of $X^\#$ so that 
$B = B(p; r)$ and $B' = B(p'; r')$.  We say that $B$ is \emph{formally included} in $B'$
if $d(p,p') + r < r'$.  For the sake of simplifying exposition, we are deliberately obscuring 
the fact that the formal inclusion relation is a relation on \emph{codes} for rational open balls
rather than on the open rational balls themselves.  If $X^\#$ is computable, then 
the formal inclusion relation is c.e.

We now adapt some standard terminology from computable structure theory to the setting 
of $C^*$ algebras.  Fix a $C^*$-algebra $A$.  
The \emph{degree of categoricity} of $A$ is the least powerful Turing degree 
that computes an isometric isomorphism between any two computable presentations of $A$.  
$A$ is \emph{computably categorical} if its degree of categoricity is $\mathbf{0}$.
The \emph{computable dimension} of $A$ is the number of computable presentations of $A$
up to computable isometric isomorphism.

If $d$ is a metric on $X$, then for all $p \in X$ let 
$\vec{d}_p(q) = d(p,q)$.
The following is essentially proved in \cite{BEFGHMMT.2024+}.

\begin{theorem}\label{thm:pres.X}
Suppose $X$ is a compact Polish space and 
$C^*(X)^\#$ is a computable presentation.  Then, there is a computably 
compact presentation $X^\# = (X, d, (p_n)_{n \in \N})$ of $X$ so that the 
following hold.
\begin{enumerate}
	\item $\widehat{p_n}$ is a computable functional of $C^*(X)^\#$ uniformly in $n$.
	
	\item $\vec{d}_{p_n}$ is a computable vector of $C^*(X)^\#$ uniformly in $n$.
\end{enumerate}
\end{theorem}

As we shall show later, the two conditions in the conclusion of Theorem \ref{thm:pres.X} 
have special signifcance for the 
existence and uniqueness of evaluative presentations.

The following is proven in \cite{BEFGHMMT.2024+}.

\begin{theorem}\label{thm:comp.unit}
If $A$ is a commutative unital $C^*$ algebra, then $\unit_A$ is a computable vector 
of every computable presentation of $A$.  
\end{theorem}

The following definitions are from \cite{BEFGHMMT.2024+}.
Fix an effective enumeration $(\ratpoly_n)_{n \in \N}$ of the rational $*$-polynomials 
over indeterminants $x_0, x_1. \ldots$.
If $A^\#$ is a presentation of a $C^*$-algebra $A$, then 
$\ratpoly_n[A^\#]$ denotes the vector obtained from $\ratpoly_n$ by substituting the $j$-th distinguished vector 
of $A^\#$ for $x_j$.  
Let $A$ be a $C^*$ algebra, and fix a presentation $A^\# = (A, (u_n)_{n \in \N})$.
The set $D(A^\#) = \{(r, j, r')\ :\ r,r' \in \Q \cap (0, \infty)\ \wedge\ r < \norm{\ratpoly_j[A^\#]} < r'\}$ 
is the \emph{diagram} of 
$A^\#$.  We say that $f \in \N^\N$ is a \emph{name} of $A^\#$ if 
$D(A^\#) = \{(r,j,r')\ :\ \exists k \in \N f(k) = \langle r,j,r' \rangle\}$.  
It follows that $A^\#$ is computable if and only if it has a computable name.  

Suppose $\mathcal{M} = (X,d)$ is a metric space and $\mathcal{M}^\# = (\mathcal{M}, (s_j)_{j \in \N})$
is a presentation of $\mathcal{M}$. 
The \emph{diagram} of $\mathcal{M}^\#$ is $D(\mathcal{M}^\#) = \{(r,j,k,r')\ :\ r < d(s_j,s_k) < r'\}$.
 A \emph{name} of $\mathcal{M}^\#$ is an $f \in \N^\N$ so that 
$D(\mathcal{M}^\#) = \{(r,j,k,r')\ :\ \exists t \in \N f(t) = \langle r,j,k,r' \rangle\}$.

A \emph{total boundedness function} of $\mathcal{M}^\#$ is an
$f \in \N^\N$ so that for each $j \in \N$, $f(j)$ is a code of a 
finite $F \subseteq \N$ so that $X = \bigcup_{k \in F} B(s_k; 2^{-j})$.  
A presented metric space is compact if and only if it has a total boundedness function.

\section{Preliminaries}\label{sec:prelim}

Let $(\ )$ be a computable bijection of $\N$ onto $\N^2$. 
Let $(\ )_0$, $(\ )_1$ be the inverses of 
$(\ )$.  That is, $(n) = ( (n)_0, (n)_1)$ for all $n \in \N$.

Let $X$ be a compact Polish space.  Suppose $X^\#$ is a presentation of $X$ with metric $d_0$, and suppose 
$C^*(X)^\#$ is a presentation of $C^*(X)$.  Then, $(C^*(X) \times X)^\#$ is the presentation of the \emph{metric space}
$C^*(X) \times X$ whose $n$-th special point consists of the $(n)_0$-th rational vector of $C^*(X)$ and the 
$(n)_1$-th special point of $X^\#$ and whose metric $d$ is defined by 
$d((u, p), (v, q)) = \max\{ \norm{u - v}, d_0(p,q)\}$. 

\begin{definition}\label{def:evaluative}
Let $X^\#$ be a presentation of a compact Polish space $X$, and let 
$C^*(X)^\#$ be a presentation of $C^*(X)$.    We say that $C^*(X)^\#$ is 
\emph{evaluative over $X^\#$} if the evaluation functional of $C^*(X)$
is a computable map from $(C^*(X) \times X)^\#$ to $\C$.
\end{definition}

Let $X^\# = (X, d, (p_n)_{n \in \N})$ be a presentation of $X$.   Suppose 
$f_0$ is the unit of $C^*(X)$ and $f_{n+1} = \vec{d}_{p_n}$.  It follows from
the Stone-Weierstrass Theorem that 
$(C^*(X), (f_n)_{n \in \N})$ is a presentation of $C^*(X)$.
We call this presentation the \emph{presentation induced by $X^\#$}. 
The following augments a result that appears in \cite{Fox.2022+}.

\begin{theorem}\label{thm:pres.c*}
Suppose $X^\# = (X, d, (p_n)_{n \in \N})$ is a computably compact presentation.  
Then, there is a computable presentation $C^*(X)^\#$ so that the following hold.
\begin{enumerate}
	\item $\widehat{p_n}$ is a computable functional of $C^*(X)^\#$ uniformly in $n$.
	
	\item $\vec{d}_{p_n}$ is a computable vector of $C^*(X)^\#$ uniformly in $n$.
\end{enumerate}
\end{theorem}

\begin{proof}
Let $A = C^*(X)$, and let $A^\#$ denote the presentation of $A$ induced by $X^\#$.  
It is shown in \cite{Fox.2022+} that $A^\#$ is computable.  By definition, 
$\vec{d}_{p_n}$ is a computable vector of $C^*(X)^\#$.  
Let $n \in \N$.  Since $X^\#$ is computable, it follows that $\hat{p}_n$ is computable on 
the rational vectors of $A^\#$ uniformly in $n$.  However, since $\norm{\hat{p_n}} = 1$, it follows that 
$\hat{p}_n$ is a computable map of $A^\#$ into $\C$ uniformly in $n$.
\end{proof}

We now show that the conditions in the conclusion of Theorem \ref{thm:pres.c*}
guarantee the evaluativeness of a presentation.

\begin{lemma}\label{lm:eval.1}
Suppose $X^\# = (X, d, (p_n)_{n \in \N})$ is a computable presentation of 
a compact Polish space $X$, 
and let $C^*(X)^\#$ be a computable presentation of $C^*(X)$.  
Further, suppose the following.
\begin{enumerate}
	\item $\widehat{p_n}$ is a computable functional of $C^*(X)^\#$ uniformly in $n$.
	
	\item $\vec{d}_{p_n}$ is a computable vector of $C^*(X)^\#$ uniformly in $n$.
\end{enumerate}
Then, $C^*(X)^\#$ is evaluative over $X^\#$.
\end{lemma}

\begin{proof}
Let $A = C^*(X)$, and let $A^+ = (A, (f_n)_{n \in \N})$ be the presentation of $A$ induced by $X^\#$.  

Let $(\rho_n)_{n \in \N}$ be an effective enumeration of the rational vectors of $A^+$.
Since $d(p_n,p_m) = \widehat{p_m}(\vec{d}_{p_n})$, the hypotheses guarantee that 
$f_n$ is computable on the distinguished points of $X^\#$ uniformly in $n$.  
The identity map on $\N$ is a modulus of continuity for each $f_n$.  
Thus, $f_n$ is a computable map from $X^\#$ to $\C$ uniformly in $n$.
It then follows that $\rho_n$ is a computable map from $X^\#$ to $\C$ uniformly in $n$.

Suppose we are given an $A^+$-name of $f$ and a name of $z \in \C$.  Let $k \in \N$.
We can then compute a rational vector $\rho$ of $A^+$ so that 
$\norm{\rho - f} < 2^{-(k+1)}$.  Since $\rho$ is a computable map from $X^\#$ to $\C$, 
we can then compute a rational point $w$ of $\C$ so that $|\rho(p) - w| < 2^{-(k+1)}$.
Hence, $|f(p) - w| < 2^{-k}$.  All this is to say that from the given names of $f$ and $z$
it is possible to compute a name of $f(p)$.  Thus, $A^+$ is evaluative over $X^\#$.

Since the identity map is a computable function from $A^\#$ onto $A^+$, 
it follows that $A^\#$ is evaluative over $X^\#$.  
\end{proof}

We now show that the first condition of the conclusion of Theorem \ref{thm:pres.c*}
implies the second condition along with a certain uniqueness property.

\begin{theorem}\label{thm:eval.2}
Let $C^*(X)^\#$ be a computable presentation of $C^*(X)$.  
Suppose $X^\# = (X,d,(p_n)_{n \in \N})$ is a computable presentation of 
$X$ so that $\widehat{p_n}$ is a computable functional of $C^*(X)^\#$ uniformly 
in $n$.  Then:
\begin{enumerate}
	\item $\vec{d}_{p_n}$ is a computable vector of $C^*(X)^\#$ uniformly in $n$.  
	
	\item If $X^+ = (X, d', (q_n)_{n \in \N})$ is a computably compact presentation of $X$ so that 
	$\vec{d'}_{q_n}$ is a computable vector of $C^*(X)^\#$ uniformly in $n$, then 
	$X^+$ and $X^\#$ are computably homeomorphic.
\end{enumerate}
\end{theorem}

\begin{proof}
Let $A = C^*(X)$.  
By Theorem \ref{thm:pres.c*}, there \emph{is} a computably compact presentation 
$X^+ = (X, d', (q_n)_{n \in \N})$ of $X$ so that 
	$\vec{d'}_{q_n}$ is a computable vector of $C^*(X)^\#$ uniformly in $n$.

Let $A^+ = (A, (f_n)_{n \in \N})$ be the presentation of $A$ induced by $X^+$.
  Since $d'(q_n, p_m) = \widehat{p_m}(\vec{d'}_{q_n})$, 
it follows that $\vec{d'}_{q_n}$ is computable on the distinguished points of $X^\#$
uniformly in $n$.  Again, the identity map on $\N$ is a modulus of continuity for 
each $f_n$, and so $f_n$ is a computable map of $X^\#$ into $\C$ uniformly in $n$.

We now claim that the identity map is a computable map of $X^\#$ to $X^+$.
It suffices to show $(p_n)_{n \in \N}$ is a computable sequence of $X^+$.  
Since the identity map is a computable map of $A^+$ to $A^\#$, 
$\widehat{p_n}$ is a computable functional of $A^+$ uniformly in $n$.  
Hence, $d'(p_n, q_m) = \widehat{p_n}(\vec{d'}_{q_m})$ is computable uniformly in $m,n$.
It now follows that $p_n$ is a computable point of $X^+$ uniformly in $n$.  

It now follows that $X^\#$ is computably compact.
We now show $\vec{d}_{p_n}$ is a computable vector of $A^+$ uniformly in $n$.
Let $(\rho_n)_{n \in \N}$ be an effective enumeration of the rational vectors of 
$A^+$.  Since $f_n$ is a computable map of $X^\#$ to $\C$ uniformly in $n$, 
it follows that $\rho_n$ is a computable function of $X^\#$ into $\C$ uniformly in $n$.
Since $X^\#$ is computably compact, $\norm{\vec{d}_{p_n} - \rho_m}$ is computable
uniformly in $m,n$.  It follows that $\vec{d}_{p_n}$ is a computable vector of $A^+$.

Since the identity map is a computable function from $A^+$ to $A^\#$, 
it now follows that $\vec{d}_{p_n}$ is a computable vector of $A^\#$ uniformly in $n$.
\end{proof}

We now show that if $C^*(X)^\#$ and $X^\#$ satisfy the second condition of the conclusion of 
Theorem \ref{thm:pres.c*}, then $C^*(X)^\#$ is essentially identical to the 
presentation induced by $X^\#$.  

\begin{proposition}\label{prop:induced.comp}
Let $X^\# = (X, d, (p_n)_{n \in \N})$ be a presentation of a compact Polish space $X$.
If $C^*(X)^\#$ is a computable presentation, and if 
$\vec{d}_{p_n}$ is a computable vector of $C^*(X)^\#$ uniformly in $n$, then the 
presentation induced by $X^\#$ is computable and is isometrically isomorphic to $C^*(X)^\#$
via the identity map.
\end{proposition}

\begin{proof}
Let $A = C^*(X)^\#$.  
Suppose $A^\#$ is a computable presentation so that $\vec{d}_{p_n}$ is a computable vector of 
$A^\#$ uniformly in $n$.  Set $A = C^*(X)$, and let $A^+ = (A, (f_n)_{n \in \N})$ be the presentation 
induced by $X^\#$.  By Theorem \ref{thm:comp.unit}, $\unit_A$ is a computable vector of 
$A$.  Thus, $f_n$ is a computable vector of $A^\#$ uniformly in $n$.  Let $(\rho_n)_{n \in \N}$ be an 
effective enumeration of the rational vectors of $A^\#$.  Since the operations of $A$ are computable
operators of $A^\#$, it follows that $\rho_n$ is a computable vector of $A^\#$ uniformly in $n$.  
Hence, $\norm{\rho_n}$ is computable uniformly in $n$.  Thus, $A^+$ is computable.

It now follows that the identity map is a computable function from $A^\#$ to $A^+$.
\end{proof}

\section{Proofs of the main theorems}\label{sec:proofs.main}

The main theorems are simple consequences of the results in Section \ref{sec:prelim}.

\begin{proof}[Proof of Theorem \ref{thm:main.1}]
The existence of a presentation of $X$ over which $C^*(X)^\#$ is evaluative 
is essentially shown in \cite{BEFGHMMT.2024+}.  We note that it also follows
from Theorem \ref{thm:pres.X} and Lemma \ref{lm:eval.1}.  
The uniqueness follows from Theorem \ref{thm:eval.2}.
\end{proof}

\begin{proof}[Proof of theorem \ref{thm:main.2}]
This follows from Theorem \ref{thm:eval.2}.
\end{proof}

\begin{proof}[Proof of Theorem \ref{thm:main.3}]
The existence of such a presentation follows from Theorem \ref{thm:pres.c*} and 
Lemma \ref{lm:eval.1}.  
Let $A = C^*(X)$, and suppose $A^{\#_1}$ and $A^{\#_2}$ are evaluative over 
$X^\#$.  Thus, $\widehat{p_n}$ is a computable functional of $A^{\#_j}$ uniformly in $n,j$.
So, by Theorem \ref{thm:eval.2}, $\vec{d}_{p_n}$ is a computable vector of 
$A^{\#_j}$ uniformly in $n,j$.  

By Proposition \ref{prop:induced.comp}, 
$A^{\#_j}$ is computably isometrically isomorphic to the presentation of $A$
induced by $X^\#$. 
\end{proof}

\section{Uniformity and non-uniformity results}\label{sec:unif}

By inspection, the proof of Lemma \ref{lm:eval.1} and Theorem \ref{thm:main.2} are almost
fully uniform; the only obstacle is their dependence on 
Theorem \ref{thm:comp.unit} whose proof is not uniform.  
Thus, fully uniform versions can by obtained by adding a name of the unit as parameter.
This is also the situation with regards to the proof of Theorem \ref{thm:pres.X} in 
\cite{BEFGHMMT.2024+}.  
By combining these observations, we obtain the following.

\begin{theorem}\label{thm:unif.1}
Suppose $A = C^*(X)$.  
From a name of a presentation $A^\#$ and a $A^\#$-name of $\unit_A$, 
it is possible to compute a name of the presentation of $X$ over which $A^\#$ is evaluative
and a total boundedness function for this presentation.
Furthermore, given names of the same presentations, the computed names refer to presentations 
of $X$ that are computably homeomorphic via the identity map.
\end{theorem}

Our proof of Theorem \ref{thm:pres.c*} shows that it is fully uniform, and so
we obtain the following.

\begin{theorem}\label{thm:unif.2}
From a name of a presentation $X^\#$ of a compact Polish space and a total boundedness function for 
$X^\#$, it is possible to compute a name for the presentation of $C^*(X)$ induced by $X^\#$.  Furthermore, 
given two names for the same presentation of $X^\#$, the computed names refer to presentations of 
$C^*(X)$ that are computably isometrically isomorphic via the identity map.
\end{theorem}

We now turn to the necessity of the name for the unit as a parameter in the above.  
Suppose $A$ is a unital $C^*$ algebra.  We say that $A$ is \emph{computably unital}
if $\unit_A$ is a computable vector of every computable presentation of $A$.
We say that $A$ is \emph{uniformly computably unital} if 
there is an algorithm that given an index of a presentation $A^\#$ of $A$
computes an $A^\#$-index of $\unit_A$.  Our first step towards demonstrating the 
necessity of the unit is the following.

\begin{theorem}\label{thm:nucu}
There is a commutative unital $C^*$ algebra $A$ that is computably presentable but not uniformly 
computably unital.
\end{theorem}

\begin{proof}
For each $n \in \N$, let $J_n = [2^{-n} - 2^{-(n+2)}, 2^{-n} + 2^{-(n+2)}]$.  Let 
$X = \bigcup_n J_n \cup \{0\}$. 
Thus, $X$ has a computably compact presentation. 
Let $A = C^*(X)$. By Theorem \ref{thm:pres.c*}, $A$ is computably presentable.  

By way of contradiction, suppose $A$ is uniformly computably unital.  Thus, there is a partial
computable function $\psi$ so that for every $e \in \N$, 
if $e$ indexes a computable presentation $A^\#$ of 
$A$, then $\psi(e)$ is an $A^\#$-index of $\unit_A$.  This means that if $e$ indexes a computable
presentation $A^\# = (A, (v_n)_{n \in \N})$ of $A$, then for every $k \in \N$, 
\[
\norm{\ratpoly_{\phi_{\psi(e)}(k)}[A^\#] - \unit_A} < 2^{-k}.
\]
Let $\gamma(e) = \phi_{\psi(e)}(1)$, and let $\gamma_t(e) = \phi_{\psi_t(e),t}(1)$.

Let $s_n(t) = t^n$ for each $t \in [0, \frac{5}{4}]$, and let 
 $\widehat{s}_n = s_n |_X$.  Thus, $(\widehat{s}_n)_{n \in \N}$ is linearly dense in 
$A$.    Let $\rho_n$ denote the $n$-th rational vector of $(A, (\widehat{s}_n)_{n \in \N})$; 
in particular, we set $\rho_n = \ratpoly_n[(A, (\widehat{s}_n)_{n \in \N})]$.

Let $e \in \N$.  We define a sequence $(u^e_n)_{n \in \N}$ of vectors of $A$ as follows.  
If there is no $t \in \N$ so that $\gamma_t(e)\downarrow$ and so that 
$\norm{\rho_{\gamma(e)} - \unit_A} < 2^{-1}$, then let $u^e_n = \widehat{s}_n$ for all $n$.  
Otherwise, let $t_e$ be the least number so that $\gamma_{t_e}(e)\downarrow$.  
For each $n \in \N$, let $\Ind_n$ be the indicator function of $\bigcup_{m \leq n} J_m\ \cup\ \{0\}$.  
Let $k_e$ be the least number so that $\norm{\rho_{\gamma(e)}} = \norm{\rho_{\gamma(e)} \cdot \Ind_{k_e}}$.
Define
\[
u^e_n = \left\{
\begin{array}{cc}
\widehat{s}_n \cdot \Ind_{k_e + t_e} & n \leq \gamma(e) + t_e\\
\widehat{s}_{n - \gamma(e) - t_e - 1} & \mbox{otherwise}.\\
\end{array}
\right.
\]
For each $e \in \N$, let $G_e = (A, (u^e_n)_{n \in \N})$.  Thus, $G_e$ is a presentation of $A$.  

We claim that an index of $G_e$ can be computed from $e$.  
Specifically, let $f(e)$ be an index of the 
following procedure.  Given $k,n \in \N$, compute
$q \in \Q$ as follows.  First, compute the smallest $m$ so that each variable of $\ratpoly_n$ belongs to 
$\{x_0, \ldots, x_m\}$.  Then, compute $\ell \in \N$ so that 
$\norm{\ratpoly_n(\widehat{s}_0, \ldots, \widehat{s}_m)} = 
\norm{\ratpoly_n(\widehat{s}_0\cdot \Ind_\ell, \ldots, \widehat{s}_m\cdot \Ind_\ell)}$.  
If $\gamma_{\ell + m}(e)\uparrow$, then choose $q$ to be a rational number so that 
$|q - \norm{\ratpoly_n(\widehat{s}_0, \ldots, \widehat{s}_m)} | < 2^{-k}$.  
Otherwise, compute $q \in \Q$ so that $|q - \ratpoly_n(u^e_0, \ldots, u^e_m) | < 2^{-k}$. 

To verify that this procedure computes the norm on the rational vectors of $G_e$, 
it suffices to consider the case where $\gamma_{\ell + m}(e)\uparrow$.  
If $\gamma(e)\uparrow$, then $\widehat{s}_n = u^e_n$ for each $n \in \N$.  
So, suppose $\gamma(e)\downarrow$.  Then, $t_e > \ell + m$.  
Thus, $u^e_j = \widehat{s}_j\cdot \Ind_{k_e + t_e}$
for all $j \leq m$.  Hence, 
$\norm{\ratpoly_n(u^e_0, \ldots, u^e_m) } = \norm{\ratpoly_n(\widehat{s}_0, \ldots, \widehat{s}_m)}$.  

By the Recursion Theorem, there exists $e_0 \in \N$ so that $\phi_{e_0} = \phi_{f(e_0)}$.  
Therefore, $e_0$ indexes $G_{e_0}$.  Also, $\gamma(e_0)\downarrow$.  
 Let $p$ be the least number so that each variable that appears in $\ratpoly_{\gamma(e_0)}$ 
belongs to $\{x_0, \ldots, x_p\}$.  Thus, $p < \gamma(e_0)$.  
  It now follows that for each $j \leq p$, 
$u^{e_0}_j = \widehat{s}_j \cdot \Ind_{k_{e_0} + t_{e_0}}$.  Therefore, 
$\norm{\sum_{j \leq p} \alpha_j u^{e_0}_j - \unit_A} = 1$ which is a contradiction.
\end{proof}

We can now demonstrate the necessity of a name of the unit as a parameter.  
Moreover, we can show there is a particular space $X$ for which this parameter is necessary.

\begin{corollary}\label{cor:not.unif}
There is a a compact Polish space $X$ for which there is no algorithm that given an index of a presentation
$C^*(X)^\#$ produces indices of a presentation $X^\#$ and a total boundedness function for $X^\#$ so that 
$C^*(X)^\#$ is evaluative over $X^\#$.
\end{corollary}

\begin{proof}
By Theorem \ref{thm:nucu}, there is a compact Polish space $X$ so that $C^*(X)$ is 
computably presentable but not 
uniformly computably unital.  By Theorem \ref{thm:unif.2}, from an index of a presentation $X^\#$ 
and an index of a total boundedness function for $X^\#$, it is possible to compute an index of the 
presentation induced by $X^\#$.  By definition, the unit of $C^*(X)$ is the 
$0$-th distinguished vector of the presentation induced by any presentation $X^\#$.  By Proposition 
\ref{prop:induced.comp}, 
if $C^*(X)^\#$ is evaluative over $X^\#$, and if $C^*(X)^+$ is the presentation of $C^*(X)$ induced by 
$X^\#$,  then the identity map is a computable isometric isomorphism of $C^*(X)^+$ to $C^*(X)^\#$.  
The corollary follows. 
\end{proof}

\section{Applications}\label{sec:appl}

\subsection{Computing spatial realizations and composition operators}

In this section, we assume $X_0$ and $X_1$ are compact Polish spaces.
We also assume 
$C^*(X_j)^\#$ and $X_j^\# = (X_j, d^j, (p_{j,n})_{n \in \N})$ are 
computable presentations of $C^*(X_j)$ and $X_j$ respectively.  
Further, we assume $C^*(X_j)^\#$ is evaluative over $X_j^\#$.

\begin{theorem}[Effective Banach-Stone Theorem for $C^*(X)$]\label{thm:banach.stone}
Let 
$T$ be a computable isometric isomorphism of $C^*(X_0)^\#$ to 
$C^*(X_1)^\#$.  Then, the spatial realization of $T$ is a computable homeomorphism
of $X_1^\#$ to $X_0^\#$.  
\end{theorem}

\begin{proof}
Let $A_j = C^*(X_j)$, and let $\psi : X_1 \rightarrow X_0$ denote the spatial realization of $T$.  

Suppose a name $\Lambda$ of $p \in X_1$ is given.  Since 
$X_1^\#$ is induced by $C^*(X_1)^\#$, from $\Lambda$ we can compute 
$\widehat{p}$ (that is, from $\Lambda$ and a $A_1^\#$-name of $f$ we can compute
a name of 
$\widehat{p}(f)$).  Thus, we can also compute $\widehat{p} \circ T$.  However, 
$\widehat{p} \circ T = \widehat{\psi(p)}$.   Let $g_n(q) = \vec{d^0}_{p_{0,n}}$.
By Theorem \ref{thm:eval.2}, $g_n$ is a computable vector of $A_0^\#$
uniformly in $n$.  Therefore, from $\Lambda$, $n$ we can compute 
$\widehat{\psi(p)}(g_n) = d_0(\psi(p), p_{0,n})$.   It follows that from $\Lambda$
we can compute an $X_0^\#$-name of $\psi(p)$.
\end{proof}

\begin{proposition}\label{prop:comp.comp.op}
  Suppose $\psi$ is a computable map of 
$X_1^\#$ to $X_0^\#$.  Then, the composition map $f \mapsto f \circ \psi$ is 
a computable map of $C^*(X_0)^\#$ to $C^*(X_1)^\#$.
\end{proposition}

\begin{proof}
Let $A_j = C^*(X_j)$.  For all $f \in A_0$, let $T(f) = f \circ \psi$. 

Since $\norm{T} \leq 1$, it suffices to show $T$ is computable on the rational
vectors of $A_0^\#$.  Let $f$ be a rational vector of $A_0^\#$, and let 
$k \in \N$.  For each rational vector $g$ of $A_1^\#$, let 
$\sigma_g(p) = |g(p) - f(\psi(p))|$ for all $p \in X_1$.  Since $A_j^\#$ is evaluative over  
$X_j^\#$, it follows that $\sigma_g$ is a computable map from $A_1^\#$ to $\R$ uniformly
in $g$.  As $X_1^\#$ is computably compact, we can also conclude that 
$\norm{g - f\circ \psi} = \max_{p \in X_1} \sigma_g(p)$ is computable uniformly in $g$.
Thus, by a search procedure, we can compute a rational vector $g$ of $A_0^\#$
so that $\norm{g - T(f)} < 2^{-k}$.
\end{proof}

\subsection{Computable categoricity and related topics}

The following is immediate from Theorem \ref{thm:banach.stone} and 
Proposition \ref{prop:comp.comp.op}.

\begin{corollary}\label{cor:comp.cat.main}
Suppose $X$ is a compact Polish space.  Then, the following are equivalent.
\begin{enumerate}
	\item $C^*(X)$ is computably categorical.
	
	\item Any two computably compact presentations of $X$ are 
	computably homeomorphic.
\end{enumerate}
\end{corollary}

A. Fox has shown that $C^*([0,1])$ is not computably categorical \cite{Fox.2022+}.  We thus obtain the following.

\begin{corollary}\label{cor:pres.C01}
There exist computably compact presentations of $[0,1]$ that are not computably 
homeomorphic.
\end{corollary}

\begin{corollary}\label{cor:C01.eval}
There is a computable presentation of $C^*([0,1])$ that is not evaluative over 
the standard presentation of $[0,1]$.
\end{corollary}

A. Melnikov has shown that the \emph{metric space} $2^\omega$ is 
computably categorical \cite{Melnikov.2013}; that is, any two computable presentations of this metric 
space are computably \emph{isometric}.  
We augment this result as follows.

\begin{proposition}\label{prop:comp.cat.cantor}
Any two computably compact presentations of $2^\omega$ are computably homeomorphic.
\end{proposition}

\begin{proof}
Let $X = 2^\omega$.  Let $d_0$ denote the standard metric on $2^\omega$.  
For each $\sigma \in 2^{< \omega}$, let 
$[\sigma] = \{f \in 2^\omega\ :\ \sigma \subset f\}$, and let 
$\sigma\uparrow = \{\tau \in 2^{< \omega}\ :\ \sigma \subseteq \tau\}$.

Suppose $X^\# = (2^\omega, d, (p_n)_{n \in \N})$ is a computably compact presentation.
It suffices to show that $X^\#$ is a computably homeomorphic to the standard presentation of $2^\omega$.

We define a computable sequence $(h_s)_{s \in \N}$ of maps that satisfy the following conditions.
\begin{enumerate}
    \item The domain of $h_s$ is a finite covering of $2^\omega$ by pairwise disjoint open rational balls
    of $2^\omega$ whose diameters do not exceed $2^{-s}$.\label{c1}

    \item $\dom(h_{s+1})$ is a refinement of $\dom(h_s)$.  That is, for each $B \in \dom(h_{s+1})$ 
    there exists $B' \in \dom(h_s)$ so that $B \subseteq B'$.\label{c2}

    \item $\ran(h_s)$ is a finite covering of $2^\omega$ by pairwise disjoint open rational balls of 
$X^\#$ whose diameters do not exceed $2^{-s}$.  \label{c3}

    \item If $B \in \dom(h_s)$ and $B' \in \dom(h_{s+1})$ are such that $B \subseteq B'$, then 
    $h_s(B) \subseteq h_{s+1}(B')$.  \label{c4}
\end{enumerate}
To begin, we define $h_s$ as follows.  Using the computable compactness of $X^\#$, we search for 
rational open balls $B_1, \ldots, B_t$ of $X^\#$ that satisfy that following.
\begin{itemize}
    \item If $j \neq k$, then $\inf\{d(f,g)\ :\ f \in B_j\ \wedge\ g \in B_k\} > 0$. 

    \item $2^\omega = \bigcup_j B_j$.

    \item The formal diameter of $B_j$ is smaller than $2^{-s}$.
\end{itemize}
We then choose pairwise incomparable $\sigma_1, \ldots, \sigma_t \in 2^{<\omega}$ so that 
$2^\omega = \bigcup_j [\sigma_j]$.  Set $h_0([\sigma_j]) = B_j$.

Now, suppose $h_s$ has been defined and satisfies conditions (\ref{c1}) - (\ref{c3}) above.  
Suppose $\dom(h_s) = \{[\sigma_1], \ldots, [\sigma_t]\}$, 
and set $B_j = h_s([\sigma_j])$.  Each $B_j$ is a clopen subset of $2^\omega$.  
Moreover, each $B_j$ is a c.e. closed set of $X^\#$.  For, a rational open ball $B$ of 
$X^\#$ intersects $B_j$ if and only if there is a rational open ball $B'$ of $X^\#$ so that 
$B'$ is formally included in $B$ and 
$\inf\{d(f,g)\ :\ f \in B'\ \wedge\ g \in B_k\} > 0$ when $k \neq j$.  
Thus, each $B_j$ is a computably compact set of $X^\#$.  Thus, fixing $j \in \{1, \ldots, t\}$
for the moment, we search for and find rational open balls $B_{j,1}, \ldots, B_{j, r_j}$ 
of $X^\#$ that satisfy the following.
\begin{itemize}
    \item $\inf\{d(f,g)\ :\ f \in B_{j,k}\ \wedge\ g \in B_{j'}\} > 0$ when $j \neq j'$.

    \item $B_j = \bigcup_k B_{j,k}$.

    \item The formal diameter of $B_{j,k}$ is smaller than $2^{-(s+1)}$.  

    \item $\inf\{d(f,g)\ :\ p \in B_{j,k}\ \wedge\ g \in B_{j,k'}\} > 0$ 
    when $k \neq k'$.
\end{itemize}
We then select pairwise incomparable $\sigma_{j,1}, \ldots, \sigma_{j,r_j} \in \sigma_j\uparrow$ so that 
$[\sigma_j] = \bigcup_k [\sigma_{j,k}]$.  We then set 
$h_s([\sigma_{j,k}]) = B_{j,k}$.  

By inspection, $(h_s)_{s \in \N}$ satisfies conditions (\ref{c1}) - (\ref{c4}) above. 
When $f \in 2^\omega$, by Cantor's Theorem 
$\bigcap\{h_s(B)\ :\ s\in \N\ \wedge\ f \in B \in \dom(h_s)\}$ consists of a single point; 
define $\phi(f)$ to be this point.  It follows that $\phi$ is a bijective computable map of 
$2^\omega$ to $X^\#$.  
\end{proof}

\begin{corollary}\label{cor:2.omega}
$C^*(2^\omega)$ is computably categorical.
\end{corollary}

By contrast, T. Thewmorakot has shown that the Banach space $C(2^\omega; \R)$ is not computably categorical 
\cite{Thewmorakot.2023}.

By inspection, all of the proofs in sections \ref{sec:prelim} and \ref{sec:proofs.main} relativize.
In addition, the proof of Theorem \ref{thm:comp.unit}, though not uniform, relativizes.
These observations lead us to the following.

\begin{theorem}\label{thm:deg.cat}
Let $\mathbf{d}$ be a Turing degree.
Then, for every compact Polish space $X$, the following are equivalent.
\begin{enumerate}
    \item $\mathbf{d}$ is the degree of categoricity of $C^*(X)$.  

    \item $\mathbf{d}$ is the least powerful Turing degree that can compute 
    a homeomorphism between any two computably compact presentations of $X$.
\end{enumerate}
\end{theorem}

The following is a direct consequence of Theorem \ref{thm:banach.stone}.

\begin{corollary}\label{cor:comp.dim}
For every compact Polish space $X$, the following are equivalent.
\begin{enumerate}
    \item $C^*(X)$ has computable dimension $n$.

    \item There are, up to computable homeomorphism, exactly $n$ computable 
    presentations of $X$.
\end{enumerate}
\end{corollary}

\subsection{The classification problem for unital commutative $C^*$ algebras}

By Corollary 4.28 of \cite{Downey.Melnikov.2023}, we derive the following.

\begin{theorem}\label{thm:isom.index}
The set of all pairs $(e,e') \in \N^2$ so that $e$ and $e'$ index presentations of 
isomorphic unital commutative $C^*$ algebras is $\Sigma^1_1$-complete.
\end{theorem}

Theorem \ref{thm:isom.index} can be interpreted as saying that there is no classification 
of the separable unital commutative $C^*$ algebras.  While not surprising, it is 
nevertheless nice to have a rigorous proof.  In addition, we are unaware of any way to 
prove this fact except by the results in this paper.  

\section{Conclusion}\label{sec:concl}

Our results complement the effective Gelfand duality proven in \cite{BEFGHMMT.2024+} in that 
we can now reduce questions about the computable isomorphism of presented unital commutative $C^*$ algebras
to questions about computable homeomorphisms of presented compact Polish spaces.  
Our results depend considerably on the results in \cite{BEFGHMMT.2024+}, 
and in turn the results in \cite{BEFGHMMT.2024+} depend crucially on the availability of computable
multiplication and involution in computable presentations of $C^*(X)$.  Is the computability of these
operations necessary?  For example, if $C(X)^\#$ is a computable presentation of the Banach space
$C(X)$, does it follow that there is a presentation of $X$ over which $C(X)^\#$ is evaluative?
We conjecture the answer is `no'.  As evidence, we cite the recently announced result of Melnikov and Ng
that there is a compact Polish space $X$ so that $C^*(X)$ has a computable presentation but 
$X$ does not have a computably compact presentation.

\section{Acknowledgement}

I thank Peter Burton and Kostya Slutsky for helpful conversations.

\bibliographystyle{amsplain}
\bibliography{thisbib}

\providecommand{\bysame}{\leavevmode\hbox to3em{\hrulefill}\thinspace}
\providecommand{\MR}{\relax\ifhmode\unskip\space\fi MR }
\providecommand{\MRhref}[2]{%
  \href{http://www.ams.org/mathscinet-getitem?mr=#1}{#2}
}
\providecommand{\href}[2]{#2}
\begin{thebibliography}{10}

\bibitem{Baz.H-Trainor.Melnikov.2021+}
Nikolay Bazhenov, Matthew Harrison-Trainor, and Alexander Melnikov,
  \emph{Computable stone spaces}, 2021.

\bibitem{Brattka.Hertling.2021}
Vasco Brattka and Peter Hertling (eds.), \emph{Handbook of computability and
  complexity in analysis}, Theory and Applications of Computability, Springer,
  Cham, [2021] \copyright 2021.

\bibitem{Brown.McNicholl.2020}
Tyler Brown and Timothy~H. McNicholl, \emph{Analytic computable structure
  theory and {$L^p$}-spaces part 2}, Arch. Math. Logic \textbf{59} (2020),
  no.~3-4, 427--443.

\bibitem{Brown.McNicholl.Melnikov.2020}
Tyler~A. Brown, Timothy~H. McNicholl, and Alexander~G. Melnikov, \emph{On the
  complexity of classifying {L}ebesgue spaces}, J. Symb. Log. \textbf{85}
  (2020), no.~3, 1254--1288.

\bibitem{BEFGHMMT.2024+}
Peter Burton, Christopher~J. Eagle, Alex Fox, Isaac Goldbring, Matthew
  Harrison-Trainor, Alexander Melnikov, Timothy~H. McNicholl, and Teerawat
  Thewmorakot, \emph{Computable gelfand duality}, Submitted. Preprint available
  at https://arxiv.org/abs/2402.16672.

\bibitem{Clanin.McNicholl.Stull.2019}
Joe Clanin, Timothy~H. McNicholl, and Don~M. Stull, \emph{Analytic computable
  structure theory and {$L^p$} spaces}, Fund. Math. \textbf{244} (2019), no.~3,
  255--285.

\bibitem{Cooper.2004}
S.~Barry Cooper, \emph{Computability theory}, Chapman \& Hall/CRC, Boca Raton,
  FL, 2004.

\bibitem{Downey.Melnikov.2023}
Rodney~G. Downey and Alexander~G. Melnikov, \emph{Computably compact metric
  spaces}, Bull. Symb. Log. \textbf{29} (2023), no.~2, 170--263.

\bibitem{Fox.2022+}
Alec Fox, \emph{Computable presentations of ${C}^*$ algebras}, To appear in
  Journal of Symbolic Logic.

\bibitem{McNicholl.2017}
Timothy~H. McNicholl, \emph{Computable copies of $\ell^p$}, Computability
  \textbf{6} (2017), no.~4, 391 -- 408.

\bibitem{Melnikov.2013}
Alexander~G. Melnikov, \emph{Computably isometric spaces}, J. Symbolic Logic
  \textbf{78} (2013), no.~4, 1055--1085.

\bibitem{Melnikov.Ng.2014}
Alexander~G. Melnikov and Keng~Meng Ng, \emph{Computable structures and
  operations on the space of continuous functions}, Fundamenta Mathematicae
  \textbf{233} (2014), no.~2, 1 -- 41.

\bibitem{Melnikov.Nies.2013}
Alexander~G. Melnikov and Andr{\'e} Nies, \emph{The classification problem for
  compact computable metric spaces}, The nature of computation, Lecture Notes
  in Comput. Sci., vol. 7921, Springer, Heidelberg, 2013, pp.~320--328.

\bibitem{Thewmorakot.2023}
Teerawat Thewmorakot, \emph{Computability {T}heory on {P}olish {M}etric
  {S}paces}, Bull. Symb. Log. \textbf{29} (2023), no.~4, 664--664.

\end{thebibliography}

\end{document}